\documentclass[oneside,10pt]{amsart}
\usepackage{enumitem}
\usepackage{latexsym, bbm, amssymb, amsmath, wasysym, mathtools}
\usepackage[round]{natbib}
\usepackage{setspace}
\usepackage{ifpdf}
\usepackage[capposition=top]{floatrow}

\usepackage{tikz}
\usetikzlibrary{patterns,snakes,shapes}

\usepackage{accents}

\makeatletter
\def\subsection{\@startsection{subsection}{1}%
  \z@{.5\linespacing\@plus.7\linespacing}{.3\linespacing}%
  {\normalfont\scshape}}
\newcommand{\eqnum}{\refstepcounter{equation}\textup{\tagform@{\theequation}}}
\makeatother

\DeclarePairedDelimiter\abs{\lvert\,}{\,\rvert}%
\DeclarePairedDelimiter\norm{\lVert\,}{\,\rVert}%

\theoremstyle{plain}
\newtheorem{theorem}{Theorem}
\newtheorem{corollary}[theorem]{Corollary}
\newtheorem{lemma}[theorem]{Lemma}
\newtheorem{proposition}[theorem]{Proposition}
\theoremstyle{definition}

\newtheorem{remark}[theorem]{Remark}

\allowdisplaybreaks[1]

\newcommand{\R}{\mathbb{R}}

\newcommand{\N}{\mathbb{N}}

\renewcommand{\P}{\mathbb{P}}
\newcommand{\Ps}{\mathcal{P}}
\renewcommand{\S}{\mathcal{S}}
\newcommand{\E}{\mathbb{E}}
\newcommand{\F}{\mathcal{F}}
\newcommand{\M}{\mathcal{M}}

\newcommand{\C}{\mathcal{C}}

\newcommand{\B}{\mathcal{B}}
\newcommand{\Z}{\mathcal{Z}}

\newcommand{\X}{\mathcal{X}}

\newcommand{\Kn}{K^{(n)}}
\newcommand{\KN}{K^{(N)}}
\newcommand{\KNh}{\widehat{K}^{(N)}}

\newcommand{\1}{\mathbbm{1}}

\DeclareMathOperator{\Var}{Var}

\providecommand{\norm}[1]{\left\lVert#1\right\rVert}
\providecommand{\abs}[1]{\left\lvert#1\right\rvert}

\let\oldmarginpar\marginpar
\renewcommand\marginpar[1]{\-\oldmarginpar[\raggedleft\footnotesize #1]%
{\raggedright\footnotesize #1}}

\begin{document}
\medskip
{\sl Dedicated to the Memory of Moshe Shaked}
\medskip
\medskip

\title[]%
{A Central Limit Theorem for Costs in \\ Bulinskaya's Inventory Management Problem \\ When Deliveries Face Delays}
\author[]
{Alessandro Arlotto and J. Michael Steele}

\thanks{A. Arlotto:
The Fuqua School of Business, Duke University, 100 Fuqua Drive, Durham, NC, 27708.
Email address: \texttt{alessandro.arlotto@duke.edu}
}

\thanks{J. M. Steele:
Department of Statistics, The Wharton School, University of Pennsylvania, 3730 Walnut Street, Philadelphia, PA, 19104.
Email address: \texttt{steele@wharton.upenn.edu}
}



\begin{abstract}
It is common in inventory theory to consider policies that minimize
the expected cost of ordering and holding
goods or materials. Nevertheless,
the realized cost is a random variable, and, as the Saint Petersburg Paradox reminds us, the expected value does not always capture the
full economic reality of a decision problem. Here we take the classic inventory model of \citet{Bulinskaya:TPA1964}, and, by proving an
appropriate central limit theorem, we show in a reasonably rich (and practical) sense
that the mean-optimal policies are economically appropriate. The motivation and the tools are applicable to a large class of
Markov decision problems.



\end{abstract}


\maketitle



\section{Weighing Costs with More Than Means} \label{se:intro}

In many problems of operations management, the quantity of central interest is the cost $\C_n(\pi^*_n)$ that one incurs
over $n$ periods by following a policy $\pi^*_n$ that is optimal in the sense that
\begin{equation}\label{eq:MeanConditionPi}
\E[\C_n(\pi^*_n)]=\min \big\{ \, \E[\C_n(\pi)]: \pi \in \S_n \, \big\},
\end{equation}
where $\S_n$ is a specified set of feasible policies. This procedure is so commonplace that one may see nothing odd in it.
Nevertheless, just the mention of the Saint Petersburg Paradox is enough to remind
one that economic decisions based on means alone may be quite
ill founded. The realized cost $\C_n(\pi)$ incurred by a policy $\pi$ is a random variable, and, to compare this to the
cost $\C_n(\pi')$ incurred by another policy $\pi'$, one must make an economically meaningful comparison between the distributions
of two random variables.

Perhaps no individual has contributed more to our understanding of stochastic orders than
Moshe Shaked. The theory of orders on distributions was a life-long interest of his,
and this interest also encompassed applications to many fields.
Among his contribution that single out applications in economics, one can mention at least
\citet{MullerScarsiniShaked2002}, \citet{ColangeloScarsiniShaked-InsMathEcon-2005}, \citet{LiShakedChap2007} and
\citet{ShakedSordoSu-MCAP-2012}.
The two further papers of
\citet{MarshallShaked:AAP1983} and \citet{HartmanDrorShaked:GEB2000} make specific note of applications to
inventory management which is our principal focus here.
More broadly, the classic texts of \citet{Shaked-Shantikumar-1994,Shaked-Shanthikumar2007}
have served for a generation as fundamental resources for information on stochastic orders --- in their many varieties and for their many applications.

Here, we focus almost entirely on a well-studied inventory management model that was introduced by
\citet{Bulinskaya:TPA1964}, and the first observation is that it leads one into the same kind of economic quandary that one finds
with the generic mean-focused optimality criterion \eqref{eq:MeanConditionPi}.
In Section \ref{se:conclusions}, we will discuss how one
might approach this quandary using the theory of stochastic order, but,  in the main, we stick to an approach that is decidedly  more classical.

As a first cut, one should note that the
criterion \eqref{eq:MeanConditionPi} becomes economically meaningful precisely in those circumstances where
$\E[\C_n(\pi^*_n)]$ is a viable surrogate for $\C_n(\pi^*_n)$, and,
at a bare minimum, this calls for $\E[\C_n(\pi^*_n)]$ and $\C_n(\pi^*_n)$ to be probabilistically of the same order.
This condition does not fully escape the
Saint Petersburg Paradox, but it does give one some practical guidance.
Moreover, it was found in \citet{ArlottoGansSteele:OR2014} that such closeness in probability holds automatically for a large class of Markov decision problems.

To take this logic one step further, one can look to determine the asymptotic distribution of the realized cost $\C_n(\pi^*_n)$. This is
the approach taken here. Specifically, we prove that there is a central limit theorem (CLT)
for $\C_n(\pi^*_n)$ in the \citeauthor{Bulinskaya:TPA1964} inventory model where one may suffer
delivery delays according to chance.
Unfortunately,  such a CLT still does not free us entirely from the logical tensions
of the Saint Petersburg Paradox.
Nevertheless, it bring us one step closer to
a compelling practical rational for embracing the mean-optimal policy in \citeauthor{Bulinskaya:TPA1964}'s famous model.

\subsection{\citeauthor{Bulinskaya:TPA1964}'s Inventory Model: Delivery Delays Included}

In the classic dynamic inventory model of \citet{Bulinskaya:TPA1964},
one has $n$ periods and $n$ demands
$D_1, D_2, \ldots,D_n$ that are modeled as independent random variables
with a common  density $\psi$ that has support on a bounded subinterval of $[0,\infty)$.
At the beginning of each period $i \in [1:n] \equiv \{1, 2, \ldots, n\}$
one knows the current level of inventory $x$, and the main task is to decide
upon the quantity that should be ordered.
Here it is technically useful to allow $x$ to take both positive and negative values, and in the latter
case, the absolute value $|x|$ is understood to represent the level of \emph{backlogged} demand.
To underscore this possibility, we sometimes call $x$ the \emph{generalized} inventory level.

\citeauthor{Bulinskaya:TPA1964}'s model also allows for some uncertainty about the time when an order is filled. Specifically, the model posits that there is
a fixed, known parameter $q \in [0,1]$ such that after an order is placed
it is filled immediately with probability $q$, but, if it is not filled immediately, then it is guaranteed to be filled
at the beginning of the next period. The associated fulfilment events are assumed to be independent, and we account for them with
independent Bernoulli random variables $Y_1, Y_2, \ldots, Y_n$ where the event $Y_i=1$ indicates that the $i$'th order
is filled immediately.

\citeauthor{Bulinskaya:TPA1964}'s model further accounts for costs of two kinds. First, there is the \emph{cost of ordering}, and the model takes this cost to be
proportional to the size of the order. Specifically, to move the inventory
from level $x$ to the level $y\geq x$, one places an order of size $y-x$,  and one incurs a cost of $c(y-x)$
where the multiplicative constant $c>0$ is a model parameter that we assume to be known.

Second, there is the \emph{carrying cost} that one incurs either by holding some physical inventory or by managing a backlog.
To model this cost, one introduces two
non-negative constants $c_h$ and $c_p$, and one defines a
piecewise linear function $L: \R \rightarrow \R$ by setting
\begin{equation}\label{eq:Def-L}
L(z) =
\begin{cases}
    \hphantom{-}c_h z      & \quad \text{if } z\geq 0 \\
    -c_p z                  & \quad \text{if } z  < 0.
\end{cases}
\end{equation}
The interpretation here is that if $z\geq 0$, then $L(z)$ represents the cost for holding a quantity $z$ of inventory from one period to the next,
while, if $z< 0$, then $L(z)$ represents the penalty cost for managing a backlog of size $-z =|z| \geq 0$.

Given this framework, the carrying cost that is realized in period $i$
depends on whether the order is filled immediately or not.
If the inventory at the beginning of period $i$ is $x$ and
one places an order to bring the inventory up to  the level $y$, then in the case
(i)
in which the order is filled immediately the carrying cost for period $i$ is $L(y - D_i)$, but in the case
(ii)
in which the delivery is delayed to the beginning of the next period the carrying cost is $L(x - D_i)$.
Thus, the carrying cost for period $i$ can always be written as
\begin{equation}\label{eq:carrying-cost}
Y_i L(y - D_i) + (1-Y_i) L(x - D_i).
\end{equation}

Finally, we assume that the ordering cost rate $c$ is strictly smaller than the backlog management penalty rate $c_p$;
thus, it is never optimal to accrue backlog penalty costs if one has the opportunity  to place an order.

\subsection{A Mean-Optimal Policy}\label{sse:MeanOptimalPolicy}
\citeauthor{Bulinskaya:TPA1964} studied this problem in great detail,
and she characterized the structure of the ordering policy that
minimizes the total expected inventory costs over the $n$ periods. That is, she
found the optimal policy in the sense of \eqref{eq:MeanConditionPi}.

In brief, she found values $\gamma_{n,i}(x)$ such that
if the inventory at the beginning of period $i$ is $x$, then it is optimal in the sense of \eqref{eq:MeanConditionPi} to place an order
of size $\gamma_{n,i}(x)-x$ so that the order would bring the inventory up to the level $\gamma_{n,i}(x)$.
We call the map $x \mapsto \gamma_{n,i}(x)$ the \emph{order-up-to function},
and \citeauthor{Bulinskaya:TPA1964} gave a concrete formula for its value.

Specifically, she found that if one sets
\begin{equation}\label{eq:n0}
n_0 = \inf\{ j \in \mathbb{N} : c < c_p(q + j +1) \}
\end{equation}
then there is a sequence of non-decreasing real values
\begin{equation}\label{eq:base-stock-level-monotone}
0 < s_{n_0+2} \leq s_{n_0+3} \leq \cdots\leq s_n \leq s_\infty
\end{equation}
such for $i$ in the \emph{principal range}  $[1:n-n_0)\equiv \{1,2, \ldots, n-n_0-1\}$
one has the \emph{active rule}
\begin{equation}\label{eq:base-stock-policy}
\gamma_{n,i}(x) =
\begin{cases}
    s_{n-i+1}       & \quad \text{if } x \leq s_{n-i+1} \\
    x                  & \quad \text{if } x > s_{n-i+1} .
\end{cases}
\end{equation}
Instead, for $i$ in the \emph{residual range}  $[n - n_0 :  n]$, one has the \emph{passive rule}
\begin{equation}\label{eq:base-stock-policy-bis}
\gamma_{n,i}(x) = x,
\end{equation}
so, for $i$ in the residual range the optimal action is to passively order nothing.

In general, the concrete values of the levels $\{s_i: n_0+2 \leq i \leq n \}$ can only be computed numerically, but those numerical values are not needed here.
Parenthetically, we should also note that the curious starting level
$s_{n_0+2}$ comes about because it is the value that one needs to compute
to characterize $\gamma_{n,i}(x)$ when $i=n-n_0-1$;
i.e. when $i$ is the largest element of the principal range $[1:n-n_0)$.

\citeauthor{Bulinskaya:TPA1964}'s optimal restocking rules determine a natural \emph{inventory process},
where we view $X_{n,i}$ as the generalized inventory at the \emph{beginning of period} $i$, which we define to be the moment before one places the order for period $i$.
If we take the initial inventory level to be $X_{n,1}=x$,
then the restocking rules \eqref{eq:base-stock-policy} and \eqref{eq:base-stock-policy-bis}
tell us that the subsequent value of the generalized inventory level
 at the beginning of period $i+1$ is always given by
\begin{equation}\label{eq:Xi-inventory}
X_{n,i+1} = \gamma_{n,i}(X_{n,i}) - D_i \quad \quad \text{for all } i \in [1:n].
\end{equation}
In particular, the generalized inventory at the beginning of period $i+1$ does not depend on whether
the order in period $i$ was delayed or not.
In fact, a period $i$ order that is delayed is assumed to be filled by the beginning of period $i+1$.

From the recursion \eqref{eq:Xi-inventory},
we see that the sequence $\{X_{n,i}: 1 \leq i \leq n+1\}$ is a non-homogenous Markov chain,
and, if we let $\F_0$ denote the trivial $\sigma$-field and set
\begin{equation}\label{eq:SigmaField}
\F_i=\sigma\{D_1,D_2,\ldots D_i, Y_1, Y_2, \ldots, Y_i\}, \quad \quad \text{ for } i \in [1:n],
\end{equation}
then we see that  $X_{n,i}$ is $\F_{i-1}$-measurable for each $i \in [1:n+1]$.

Now, if we take $0<J<\infty$
to be the least upper bound of the support of the demand distribution $\psi$ and
if we begin with $X_{n,1}=x \in [0, s_\infty]$, then the recursion \eqref{eq:Xi-inventory} tells us that
\begin{align*}
X_{n, i} \in [- J, s_\infty] \quad &\text{for all }  i \in  [2:n-n_0) \,\, \text{and} \\
X_{n, i} \in [- (n_0+2)J, s_\infty] \quad &\text{for all } i \in  [n-n_0:n+1]. \notag
\end{align*}
The most natural starting inventory level is $X_{n,1}=0$,
and, \emph{de minimus}, we  will assume throughout that $X_{n,1}=x \in [0, s_\infty]$.
As a consequence, we can take the state space
of the generalized inventory process to be  $\X=[\,- (n_0+2)J,\, s_\infty\,]$.
So, in particular, we have that the state space $\X$ of the inventory chain is bounded.

\subsection{Realized Costs}

The restocking policy $\pi^*_n$ determined by \citeauthor{Bulinskaya:TPA1964}'s rules
\eqref{eq:base-stock-level-monotone} and \eqref{eq:base-stock-policy} is optimal in the sense that
it attains the minimal \emph{expected} inventory cost over all feasible policies.
Of course, the actual cost realized by
following the policy $\pi^*_n$ is a random variable, and it turns out to be useful to write it
as the cumulative sum of the ordering cost and the carrying costs.


The \emph{one-period cost} $\Ps_{n,i}$ for period $i$ is the sum of
(i) the ordering costs and
(ii) the holding cost that one calculates from \eqref{eq:carrying-cost}.
Thus, when we follow \citeauthor{Bulinskaya:TPA1964}'s  mean-optimal inventory policy,
the one-period cost for period $i$ is given by
\begin{equation}\label{eq:PeriodCost1}
\Ps_{n,i}=
c( \gamma_{n,i}(X_{n,i}) - X_{n,i} )
+ Y_i L (\gamma_{n,i}(X_{n,i}) - D_i) + (1-Y_i) L (X_{n,i} - D_i).
\end{equation}
Thus, the corresponding \emph{realized cost} accrued over time is
\begin{align}\label{eq:TotalCost}
\C_i(\pi^*_n) = \sum_{j=1}^i \Ps_{n,j}, \quad \quad \text{for } i \in [1:n],
\end{align}
and, of course, our main concern is with the total cost $\C_n(\pi^*_n)$ over the full time horizon.

As we noted before, the mean-focused optimality \eqref{eq:MeanConditionPi} of \citeauthor{Bulinskaya:TPA1964}'s policy
is economically meaningful only if the expected value of the
total cost $\C_n(\pi^*_n)$ can serve as a good proxy for the cost that is actually realized.
While there is no \emph{a priori} reason for this to be the case,
we can increase our confidence in the use of $\pi^*_n$
if we learn more about the distribution of the total realized cost $\C_n(\pi^*_n)$.

\subsection{A CLT Mean-Optimal for Inventory Costs}

For \citeauthor{Bulinskaya:TPA1964}'s optimal policy one can show that after the natural centering and scaling
the total realized cost $\C_n(\pi^*_n)$ is asymptotically normal ---  provided that the
demand density satisfies a natural regularity condition.

To be explicit, we say that a demand density $\psi$ is \emph{softly unimodal} provided
that for each $\epsilon \geq 0$ there is a $\widehat{w}=\widehat{w}(\epsilon)\in [-\infty,\infty)$ such that
\begin{align}\label{eq:density-condition-new}
\{w: \psi(w + \epsilon) \leq \psi(w) \} =[\widehat{w}, \infty).
\end{align}
Naturally any unimodal density is softly unimodal,
and any density with two isolated modes is not softly unimodal.
The uniform density on $[0,1]$ is a simple example of a density which is not unimodal
but which is softly unimodal.
Although the condition \eqref{eq:density-condition-new} is a little unconventional,
it is satisfied by essentially all of the
distributions that are commonly considered in inventory theory.
Moreover, it is exactly what is needed in our context.

\begin{theorem}[CLT for Mean-Optimal Inventory Cost]\label{thm:inventory-CLT}
If the demand density $\psi$ is softly unimodal and has bounded support in $[0,\infty)$, then
the inventory cost $\C_n(\pi^*_n)$ that is realized under the mean-optimal policy $\pi^*_n$
satisfies
$$
\frac{\C_n(\pi^*_n) - \E[\C_n(\pi^*_n)]}{\sqrt{\Var[\C_n(\pi^*_n)]}} \Longrightarrow N(0,1), \quad \quad \text{as } n \rightarrow \infty.
$$
\end{theorem}

We will prove this theorem over the next several sections.
In the first of these, we recall a CLT for non-homogenous Markov
chains from \citet{ArlSte:MOR2016}, where there is also a discussion of the  CLT for $\C_n(\pi^*_n)$ in
the context of the much simpler \citeauthor{Bulinskaya:TPA1964} model where one does not have delayed deliveries.

Our main arguments are distributed over Sections \ref{se:connect-Thm1-2}, \ref{se:step1},
and \ref{se:step2}.
Finally, in Section \ref{se:conclusions} we note three further steps that one can take
toward the economic justification for the use of the
mean-optimal policies in \citeauthor{Bulinskaya:TPA1964}'s model and related problems.
One of these depends on the theory of stochastic orders.

\section{Background on Non-Homogeneous Markov Chains} \label{se:CLT-review}

The proof of Theorem \ref{thm:inventory-CLT} depends  on a recent theorem for functions on the path of a finite horizon,
temporally non-homogenous Markov chain.
Here, we generically denote such a chain by $\{Z_{N,i}: 1 \leq i \leq N\}$. We also assume that the horizon $N \in \N$ is  fixed,
and the state space $\Z$ of the chain is a Borel space.

If $\B(\Z)$ denotes the class of Borel subsets of $\Z$ and $\nu$ is the initial distribution of the chain,
then we have the relation $\nu(B)=\P(Z_{N,1}\in B)$ for all $B \in \B(\Z)$.
The laws for the subsequent values of the chain  are then determined by
specifying the one-step transition kernels $\{\KN_{i,i+1} : 1 \leq i <N\}$ such that
\begin{equation*}
\KN_{i,i+1}(z,B) = \P(Z_{N,i+1} \in B \,|\, Z_{N,i} = z)
\quad \text{for }  z \in \Z \,\, \text{and }B \in \B(\Z).
\end{equation*}
While the transition kernels can be quite general, we do require some conditions.
To state these we first recall that for any
Markov transition kernel $K$ on $\Z$, the \emph{\citeauthor{Dobrushin:TPA1956} contraction coefficient} of $K$ is defined by
\begin{equation}\label{eq:contraction-coefficient-def}
    \delta(K) = \sup_{\substack{z_1, z_2 \in \Z \\ B \in \B(\Z)}} \abs{ K(z_1, B) - K(z_2, B) },
\end{equation}
and the corresponding \emph{ergodic coefficient} of $K$ is given by
\begin{equation*}
\alpha(K) = 1 - \delta(K).
\end{equation*}

Our main concern here is with functions on the paths of an array of time non-homogeneous Markov chains.
To define these functions we first fix an integer $m$ and for each $n=1,2, \ldots$ we introduce an array of $n$ real valued functions
of $1+m$ variables,
$$
\varphi_{n,i}: \Z^{1+m} \rightarrow \R, \quad i \in [1:n].
$$
If we now set $N=n+m$, then the object of central interest is the sum
\begin{equation}\label{eq:Sn}
S_n = \sum_{i=1}^n  \varphi_{n,i}(Z_{N,i}, \ldots, Z_{N,i+m}).
\end{equation}
Here, one should note that this sum with $n$ summands is a function of the $N$ variables $\{Z_{N,i}: 1 \leq i \leq N \}$ where $N=n+m$.

\subsection{A General Central Limit Theorem}

In favorable circumstances, when the sums $\{S_n: n \geq 1\}$ defined by \eqref{eq:Sn} are centered about their mean and
scaled by their standard deviations, one would naturally expect that they converge in distribution to the standard Gaussian.
\citet{ArlSte:MOR2016} proved that this is indeed the case provided that one has a modest
compatibility relation between the sizes of three quantities:
(i) the sizes of the ergodic coefficients $\alpha(\KN_{i,i+1})$,
(ii) the functions $\{\varphi_{n,i} : 1 \leq i \leq n\}$, and
(iii) the variance of $S_n$.
\citet{Dobrushin:TPA1956} and \citet{SetVar:EJP2005} had
earlier proved theorems that cover the special case where $m=0$ (or $N=n$).

To write that compatibility relation compactly for general arrays of transition kernels,
it is useful to introduce one further quantity.
For given values of $m$, $n$, and $N=n+m$, the key quantity in the compatibility relation is given by
\begin{equation*}
\alpha_n = \min_{1 \leq i < n} \alpha(\KN_{i,i+1}).
\end{equation*}
There should be no cause for confusion if we call this the
$n$'th \emph{minimal ergodic coefficient}, since, in the special case with $m=0$,
it reduces to the classical quantity of the same name.
The benefit of $\alpha_n$ is that
gives us an exceptionally useful measure of the ``randomness"
of the process determined by the $N$'th row of
the array of transition kernels. Here and subsequently we use $\norm{f}_\infty$ to denote
the essential supremum of the absolute value of the function $f$.

\begin{theorem}[\citealp{ArlSte:MOR2016}]\label{th:CLT-nonhomogeneous-chain}
If there are constants $C_1, C_2,  \ldots$ such that
\begin{equation*}
\max_{1 \leq i \leq n} \norm{\varphi_{n,i}}_\infty  \leq C_n
\quad
\text{and}
\quad
C_n^2 \alpha_n^{-2}  = o(\Var[S_n]),
\end{equation*}
then one has the convergence in distribution
\begin{equation}\label{eq:CLTLimit}
\frac{S_n - \E[S_n]}{\sqrt{\Var[S_n]}} \Longrightarrow N(0,1), \quad \quad \text{ as } n \rightarrow \infty.
\end{equation}
\end{theorem}

\begin{corollary}\label{cor:CLT}
If there are constants $c > 0$ and $C<\infty$ such that
$$\alpha_n \geq c \quad \text{and} \quad  C_n \leq C \, \text{ for all } n \geq 1,$$
then one has the asymptotic normality \eqref{eq:CLTLimit}
whenever $\Var[S_n] \rightarrow \infty$ as $n\rightarrow \infty$.
\end{corollary}

\section{Connecting Theorems  \ref{thm:inventory-CLT} and \ref{th:CLT-nonhomogeneous-chain}} \label{se:connect-Thm1-2}

The key observation is that with some restructuring
the realized inventory cost $\C_n(\pi^*_n) $
in \eqref{eq:TotalCost} can be seen as a sum of the form \eqref{eq:Sn}.
In particular, we need to construct an appropriate temporally non-homogeneous
Markov chain.

We do not apply Theorem \ref{th:CLT-nonhomogeneous-chain} to the inventory process \eqref{eq:Xi-inventory},
but rather to a related process on the enlarged state space $\Z \stackrel{\rm def}{=}\X\times \{0,1\}$
that we define by setting
\begin{equation*}
Z_{N,i} = (X_{n,i}, Y_i) \quad \quad \text{ for } 1 \leq i \leq N =  n+1.
\end{equation*}
If we now set
$$
\varphi_{n,i}((x,y),(x',y'))
=  c( \gamma_{n,i}(x) - x ) + y L (x') + (1-y) L (x - \gamma_{n,i}(x) + x'),
$$
then by \eqref{eq:PeriodCost1} the one period cost equals $\varphi_{n,i}(Z_{N,i}, \, Z_{N,i+1})$
and the total cost $\C_n(\pi^*_n)$ defined by \eqref{eq:TotalCost} has the representation
\begin{equation*}
\C_n(\pi^*_n) = \sum_{i=1}^n \varphi_{n,i}(Z_{N,i}, \, Z_{N,i+1}).
\end{equation*}
We noted in Subsection \ref{sse:MeanOptimalPolicy} that the
state space $\X$ is bounded, so by the piecewise linearity of $L$,
we see that the cost functions $\{\varphi_{n,i}: 1 \leq i \leq n\}$ are uniformly bounded.

Now, recalling the principal range $[1:n-n_0)$ and the residual range $[n-n_0:n]$ that determine the
restocking rules \eqref{eq:base-stock-policy} and \eqref{eq:base-stock-policy-bis},
we consider the decomposition
\begin{align}\label{eq:Cn-decomposition}
\C_n(\pi^*_n) &= \sum_{i=1}^{n-n_0-1} \varphi_{n,i}(Z_{N,i}, \, Z_{N,i+1}) + \sum_{i=n - n_0}^{n} \varphi_{n,i}(Z_{N,i}, \, Z_{N,i+1}) \\
&\stackrel{\rm def}{=} \C_n'(\pi^*_n)+\C_n''(\pi^*_n). \notag
\end{align}
The sum $\C_n''(\pi^*_n)$ has $n_0+1$ terms,
and  we know by \eqref{eq:n0} that $n_0$ is bounded by a constant that does not
depend on $n$. Since the
summands of $\C_n''(\pi^*_n)$ are uniformly bounded,
we see that the norm $\norm{\C_n''(\pi^*_n)}_\infty$ is also bounded independently of $n$.
Consequently, if $\Var[ \C_n'(\pi^*_n)] \rightarrow \infty$ and if $\C_n'(\pi^*_n)$ satisfies the CLT, then
we will find that
$\C_n(\pi^*_n)$ also satisfies the same CLT.

To prove the CLT for  $\C_n'(\pi^*_n)$ there are two basic steps --- neither of which is immediate.
First, one needs an appropriate lower bound on the minimal
ergodic coefficient of the chain $\{Z_{N,i} : 1 \leq i \leq n-n_0\}$.
Second, one needs to show that the variance of the first sum $\C_n'(\pi^*_n)$
grows to infinity as $n \rightarrow \infty$.

\section{Step One: Lower Bounds for Ergodic Coefficients}\label{se:step1}

Given the construction of the inventory process $\{X_{n,1}, X_{n,2}, \ldots, X_{n,n+1}\}$,
we can define one-step transition kernels on the state space $\X$ by setting
$$
\Kn_{i,i+1}(x,A) = \P(  X_{n,i+1} \in A \,|\,  X_{n,i} = x )
$$
for all $x \in \X$, $A \in \B(\X)$ and $i \in [1:n]$.
Similarly, for our newly constructed process $\{Z_{N,i}: 1 \leq i \leq N\}$
we have one-step transition kernels on $\Z=\X\times\{0,1\}$ that we can write as
$$
\KNh_{i,i+1}(z,B) = \P(  Z_{N,i+1} \in B \,|\,  Z_{N,i} = z ),
$$
where $z \in \Z$, $B \in \B(\Z)$, and $i \in [1:N)$.

If we set $B_0=B \cap \{z=(x,y): y=0\}$ and $B_1=B \cap \{z=(x,y): y=1\}$
then $B_0$ and $B_1$ are disjoint Borel subsets of $\Z$
with union $B$.
Thus, if we also set $A_0=\{x: (x, 0) \in B_0\}$ and $A_1=\{x: (x, 1) \in B_0\}$,
then by the independence assumptions from Section \ref{se:connect-Thm1-2} we have
\begin{align}\label{eq:preconvex}
\KNh_{i,i+1}(z,B)&=\KNh_{i,i+1}(z,B_0)+\KNh_{i,i+1}(z,B_1)\\
&=(1-q)\Kn_{i,i+1}(x,A_0) +q \Kn_{i,i+1}(x,A_1).\notag
\end{align}
By the decomposition \eqref{eq:preconvex} and the definition \eqref{eq:contraction-coefficient-def} of the Dobrushin coefficient,
 we now have a remarkably nice relationship:
\begin{equation}\label{eq:DobrushinRelationship}
\delta(\KNh_{i,i+1}) \leq \delta(\Kn_{i,i+1}).
\end{equation}

From \eqref{eq:DobrushinRelationship} we see that we can direct our efforts toward obtaining a
uniform upper bound on $\delta(\Kn_{i,i+1})$.
We begin this analysis with general lemma that
also explains how the soft unimodality condition \eqref{eq:density-condition-new} helps us.

\begin{lemma}\label{lm:m-TC}
If the density $\psi$ of $D_1$ satisfies the soft unimodality condition \eqref{eq:density-condition-new},
then for $\epsilon_i =|\gamma_{n,i}(x') - \gamma_{n,i}(x)|$
one has
\begin{equation}\label{eq:TV-m}
\sup_{A \in \B(\X)} \abs{ \Kn_{i,i+1}(x',A) -\Kn_{i,i+1}(x,A) } = \P(\widehat{w} \leq D_1 \leq \widehat{w}+ \epsilon_i),
\end{equation}
where $\widehat{w}=\widehat{w}(\epsilon_i)$ is the value guaranteed by the condition \eqref{eq:density-condition-new}.
\end{lemma}

\begin{proof}
Given $x \in \X$ and a Borel set $A\subseteq \X$, we introduce the Borel set
$$
A_x = \gamma_{n,i}(x) -A,
$$
so that the transition kernel $\Kn_{i,i+1}(x, A)$ can be written as
$$
\Kn_{i,i+1}(x, A)  = \P (X_{n,i+1} \in A \,|\, X_{n,i}=x ) = \P( D_{1} \in A_x ) = \int_{A_x} \psi(w) \, dw.
$$
Without loss of generality we can assume that $x \leq x'$, so the restocking formula
\eqref{eq:base-stock-policy} gives us
$\gamma_{n,i}(x) \leq \gamma_{n,i}(x')$, and for
$\epsilon_i = \gamma_{n,i}(x') - \gamma_{n,i}(x) \geq 0$
we find
$$
\Kn_{i,i+1}(x', A)  = \P (X_{n,i+1} \in A \,|\, X_{n,i}=x' ) = \P( D_{1} - \epsilon_i \in A_x ) = \int_{A_x} \psi(w+\epsilon_i) \, dw.
$$
Taking the difference of the last two displays gives us
\begin{align*}
\abs{ \Kn_{i,i+1}(x',A) -\Kn_{i,i+1}(x,A) }& = \bigg| \int_{A_x} \psi(w) \, dw - \int_{A_x} \psi(w+\epsilon_i)  \, dw \bigg| \\
& \leq \sup_{C \in \B(\R)} \bigg|\int_{C} \psi(w) \, dw - \int_{C} \psi(w+\epsilon_i) \, dw \bigg|.
\end{align*}
Here the supremum is attained at $C^*=\{ w: \psi(w+\epsilon_i)\leq \psi(w)\}$, and by the
soft unimodality  condition \eqref{eq:density-condition-new} for $\psi$, we have $C^*=[\widehat{w}, \infty)$. Hence, we have
\begin{align*}
\sup_{A \in \B(\X)} \abs{ \Kn_{i,i+1}(x',A) & -\Kn_{i,i+1}(x,A) }
\leq \int_{\widehat{w}}^\infty \big\{\psi(x) -\psi(x+\epsilon_i) \big\} \, dx \\
& = \P(D_1 \geq  \widehat{w})-\P(D_1-\epsilon_i \geq \widehat{w}) = \P(\widehat{w}\leq D_1 \leq \widehat{w}+\epsilon_i),
\end{align*}
just as needed to complete the proof of the lemma.
\end{proof}

Now we need a uniform estimate for the bounding probability in \eqref{eq:TV-m}. An explicit bound can be given
only over the principal range, and this is the reason why we work with the decomposition \eqref{eq:Cn-decomposition}.
The argument is not difficult but it requires some further facts from \citet{Bulinskaya:TPA1964}.

\begin{lemma} \label{lm:probability-bound}
If $x,x' \in \X$, $\epsilon_i =|\gamma_{n,i}(x') - \gamma_{n,i}(x)|$, and
$$
\kappa = \max \bigg\{\frac{c_p}{c_h+c_p}, \frac{(q + n_0 + 1) c_h+c}{(q + n_0 + 1)(c_h+c_p)} \bigg\},
$$
then for all $i \in [1:n-n_0)$ one has
\begin{equation*}
\sup_{w \in \R} \P( w \leq D_1 \leq w+\epsilon_i)
\leq \kappa < 1.
\end{equation*}
\end{lemma}

\begin{proof}
Without any loss of generality, we again take $x \leq x' $.
Next, by the monotonicity of the restocking formula \eqref{eq:base-stock-policy}
and the monotonicity of the optimal stocking levels  \eqref{eq:base-stock-level-monotone}
we have the bounds
\begin{equation}\label{eq:GammaBracket}
s_{n_0+2} \leq \gamma_{n,i}(x) \leq \gamma_{n,i}(x') \leq s_\infty \quad \quad \text{for all } i \in [1:n-n_0).
\end{equation}
For $\epsilon_i = \gamma_{n,i}(x') - \gamma_{n,i}(x)$ this tells us
$$
0 \leq \epsilon_i = \gamma_{n,i}(x') - \gamma_{n,i}(x) \leq s_\infty - s_{n_0+2}.
$$
If $w +\epsilon_i \leq s_\infty$, then we have the trivial bound
\begin{equation}\label{ex:DBound1}
\P(w \leq D_i \leq w+\epsilon_i) \leq \P(D_i \leq s_\infty).
\end{equation}
On the other hand, if $w +\epsilon_i \geq s_\infty$ then by \eqref{eq:GammaBracket} we get
$w \geq s_{n_0+2}$
and we again have a trivial bound,
\begin{equation}\label{ex:DBound2}
\P(w \leq D_i \leq w+\epsilon_i) \leq \P(D_i \geq s_{n_0+2}).
\end{equation}

To make use of these two bounds, we need some estimates on the stocking levels $s_{n_0+2}$ and $s_\infty$. Fortunately,
from \citet[Theorems 1 and 2]{Bulinskaya:TPA1964}  we find that
for $\Psi(t)=\P(D_1\leq t)$ one has
$$
s_{n_0+2} \geq \Psi^{-1}\left( \frac{(q+ n_0 + 1)c_p - c}{(q+ n_0 + 1)(c_p+c_h)} \right)
\quad \text{and} \quad
s_\infty \leq \Psi^{-1}\left( \frac{c_p}{c_p+c_h}\right),
$$
or, in other words,
\begin{equation}\label{eq:ProbDbelowSinf}
\P(D_i \leq s_\infty) \leq \frac{c_p}{c_h+c_p}
\quad \quad \text{and} \quad \quad
\P(D_i \geq s_{n_0+2}) \leq  \frac{(q+ n_0 + 1) c_h + c}{(q+ n_0 + 1)(c_p+c_h)},
\end{equation}
Since $c_h>0$ the first ratio is always strictly less than one.  Also, by the definition  \eqref{eq:n0}
of $n_0$ we have the strict inequality $c < (q + n_0 + 1) c_p$, so the second ratio in \eqref{eq:ProbDbelowSinf}
is strictly less than one. This gives us $\kappa<1$, and, by the bounds \eqref{ex:DBound1}, \eqref{ex:DBound2},
and \eqref{eq:ProbDbelowSinf} the proof of the
lemma is complete.
\end{proof}

By Lemmas \ref{lm:m-TC} and \ref{lm:probability-bound}, we see that for each index $i$ in the principal range $[1:n - n_0)$ we have a uniform bound on the Dobrushin coefficient:
$$
\delta ( \Kn_{i,i+1} )
= \sup_{\substack{x, x' \in \X \\ A \in \B(\X)}}
    \abs{\Kn_{i,i+1}( x, A ) - \Kn_{i,i+1}( x', A ) }
\leq \kappa.
$$
In turn, the elementary bound \eqref{eq:DobrushinRelationship},
tells us in turn that for  the minimal ergodic coefficient
$\alpha_{n - n_0 }$ of the
non-homogeneous Markov chain $\{Z_{N,i}: 1 \leq i \leq n-n_0 \}$ we have
\begin{equation*}
\alpha_{n - n_0 } \geq \min_{1 \leq i < n - n_0 } \{ 1 - \delta ( \Kn_{i,i+1} ) \}
\geq  1 - \kappa >0,
\end{equation*}
and this completes our first step in the proof of the CLT for $\C_n'(\pi^*_n)$.

\section{Step Two: Variance Lower Bound}\label{se:step2}

Here the plan is to estimate the variance of $\C_n(\pi^*_n)$ with help from a martingale that
we derive from the Bellman equation of dynamic programming.
In our case, this is a recursive equation for the values
$\{v_k(x): x \in \X \text{ and } 1\leq k \leq n\}$
that represent the expected total inventory cost over the
$k$ time periods in $[n-k+1:n]$ when
(i) we begin with an inventory level $x$ at the beginning of period $n-k+1$ and
(ii) we follow the optimal policy $\pi^*_n$ over the remaining $k$ periods.
In dynamic programming the map $x \mapsto v_k(x)$ is called the value function, or,
the \emph{cost-to-go function} since it represents the expected cost under the optimal policy
when there are $k$ remaining periods and the generalized inventory
at the beginning of period $n-k+1$ is equal to $x$.

If we now take $v_0 (x) \equiv 0$  for all $x \in \X$
as our terminal condition, then
we find by backward induction (or, equivalently, by the optimality principle of dynamic programming)
that for $k \in [1:n]$ the function $v_k(x)$ can be computed  by the recursion:
\begin{align}\label{eq:value-functions-inventory}
v_k(x) = \min_{y \geq x} \bigg\{ & c(y-x) + \E\big[Y_{n-k+1} L(y - D_{n-k+1})\big] \\
                                & + \E\big[(1-Y_{n-k+1}) L(x - D_{n-k+1})\big] + \E\big[v_{k-1}(y-D_{n-k+1})\big]\bigg\}. \notag
\end{align}
In particular, if the initial inventory level is $x$,
then one has
$$
v_n(x) = \E[\, \C_n(\pi^*_n)\,].
$$
\citet{Bulinskaya:TPA1964} established that
the minimum in \eqref{eq:value-functions-inventory} is uniquely attained at the value
$y= \gamma_{n,n-k+1}(x)$.
This is the order-up-to level for period $n-k+1$, so the inventory at the beginning of period $n-k+2$ is given by $\gamma_{n,n-k+1}(x)-D_{n-k+1}$.
Thus, from the beginning of that period forward, the expected remaining cost under the optimal strategy $\pi^*_n$ is given by
$v_{k-1}(\gamma_{n,n-k+1}(x)-D_{n-k+1})$.

The recursion \eqref{eq:value-functions-inventory} is usefully viewed as the sum of two pieces; specifically, we write
\begin{equation}\label{eq:value-functions-inventory-2}
v_k(x)=A_k(x) +B_k(x)
\end{equation}
where we define $A_k(\cdot)$ and  $B_k(\cdot)$ by setting
\begin{align}
A_k(x)&= c(\gamma_{n,n-k+1}(x)-x) +  \E[(1-Y_{n-k+1}) L(x - D_{n-k+1})], \quad \text{and} \notag \\
B_k(x)&= \E[Y_{n-k+1} L(\gamma_{n,n-k+1}(x) - D_{n-k+1})]+ \E[v_{k-1}(\gamma_{n,n-k+1}(x)-D_{n-k+1})]. \notag
\end{align}
By the analysis of \citet{Bulinskaya:TPA1964}, the order-up-to function
$\gamma_{n,n-k+1}(x)$ is given more explicitly by the
restocking rules \eqref{eq:base-stock-policy} and \eqref{eq:base-stock-policy-bis},
but the recursion \eqref{eq:value-functions-inventory-2} is independently useful. For example,
we can use it now to prove a technical lemma that we will need shortly.

\begin{lemma}\label{lm:Technical} If $x \leq 0$, $x'\leq 0$, and $n-k+1 \in [1:n-n_0)$, then one has
$$
v_k(x)-v_k(x')= (x'-x) \{ c+ c_p(1-q) \}.
$$
\end{lemma}

\begin{proof}
The crucial observation here is that for $n-k+1$ in the principal range $[1:n-n_0)$,
the restocking rule \eqref{eq:base-stock-policy} tells us that for  $x$ and $x'$ in $(-\infty, 0]$ we have
\begin{equation}\label{eq:GammaCost}
\gamma_{n,n-k+1}(x)=\gamma_{n,n-k+1}(x')=s_k.
\end{equation}
Thus, the formula for $B_k(\cdot)$ in the decomposition \eqref{eq:value-functions-inventory-2}
gives us $B_k(x)=B_k(x')$. By
the definition \eqref{eq:Def-L} of the carrying-cost function $L$, the independence of $Y_{n-k+1}$ and $D_{n-k+1}$, and a second application
of \eqref{eq:GammaCost}, we then have
\begin{align*}
v_k(x)-v_k(x')&=A_k(x)-A_x(x') \\
&=c(x'-x)+(1-q) \E[(-c_p)(x-D_{n-k+1}) + c_p(x'-D_{n-k+1})]\\
&=(x'-x) \{ c+ c_p(1-q)\},
\end{align*}
just as needed to complete the proof of the lemma.
\end{proof}

A more thematic benefit of the value functions $\{v_k: 0 \leq k \leq n \}$
is that they lead one to a useful martingale with respect to the natural filtration
$\{\F_{i}: 0 \leq i \leq n\}$
that was introduced in \eqref{eq:SigmaField}.

\begin{proposition}[Optimality Martingale]\label{pr:OptMart}
The process defined for $i \in [0:n]$ by
\begin{equation}\label{eq:Def-Opt-Mart}
\M_{n,i} =\C_i(\pi^*_n) + v_{n-i}(X_{n,i+1})=\C_i(\pi^*_n) + v_{n-i}(\gamma_{n,i}(X_{n,i}) -D_i)),
\end{equation}
is an $\{\F_{i}: 0 \leq i \leq n \}$ adapted martingale.
The martingale differences
\begin{equation}\label{eq:MartDiffs}
d_{n,i}\equiv \M_{n,i} -\M_{n,i-1} = \Ps_{n,i} + v_{n-i}(X_{n,i+1})-v_{n-i+1}(X_{n,i})
\end{equation}
then provide the representation
\begin{equation}\label{eq:MDS-rep-Cn}
\C_n(\pi^*_n) - \E[\C_n(\pi^*_n)] = \sum_{i=1}^n d_{n,i}.
\end{equation}
\end{proposition}
\begin{proof}
The $\F_{i}$-measurability of $\M_{n,i}$ is obvious, and the integrability of $\M_{n,i}$
follows from the boundedness of state space $\X$ and the piecewise linearity of $L$.
Also, the second equality in \eqref{eq:Def-Opt-Mart} is simply given by the recursive identity \eqref{eq:Xi-inventory}.

To get the martingale property, we first note
that the definition of the value function gives us
$v_{n-i}(X_{n,i+1})\equiv \E[\C_{n}(\pi^*_n) - \C_{i}(\pi^*_n) \, | \F_{i}]$,
so, from the definition of $\M_{n,i}$, we have
\begin{align*}
\M_{n,i} =\C_i(\pi^*_n) + v_{n-i}(X_{n,i+1})
    &=\C_i(\pi^*_n)+\E[ \C_n(\pi^*_n)- \C_i(\pi^*_n) \, | \,\F_{i}]\\
    &=\E[ \C_n(\pi^*_n) \,|\, \F_{i}].
\end{align*}
This confirms that the process $\{\M_{n,i}, 0 \leq i \leq n\}$
is a martingale with respect to the increasing
sequence of $\sigma$-fields $\{\F_{i}: 0\leq i \leq n\}$.
It is, in fact, a Doob martingale.
Finally, we should note that since $\M_{n,0}=\E[\C_n(\pi^*_n)]$ and $\M_{n,n}=\C_n(\pi^*_n)$, we get the
representation \eqref{eq:MDS-rep-Cn} from the
telescoping sum of the differences.
\end{proof}

\begin{remark}[Uniform Boundedness Property]\label{re:BoundedMDS}
One can further prove that there is a constant $B<\infty$ such that
\begin{equation}\label{eq:BddMDS}
\norm{ d_{n,i} }_\infty \leq B \quad  \text{for all } 1\leq i \leq n< \infty,
\end{equation}
and, as we will explain in Section \ref{se:conclusions}, this fact has some useful implications.
Nevertheless, we do not pursue the proof of \eqref{eq:BddMDS} here since it
is not needed for the proof of Theorem \ref{thm:inventory-CLT}.
\end{remark}

We now have all of the tools in place to prove a key lemma.
It will lead  in turn to the completion of the second step in
the proof of the CLT for $\C_n'(\pi_n^*)$, and, from that point, the completion of the proof of Theorem \ref{thm:inventory-CLT}
is easy.

\begin{lemma}\label{lm:VarianceLowerBound}
There are constants $\beta>0$ and $N <\infty$ such that  for all $n\geq N$, one has the variance lower bound
\begin{equation}\label{eq:VarLowerBound}
\Var[\C_n(\pi^*_n)] = \sum_{i=1}^n \E[d_{n,i}^2] \geq \beta n.
\end{equation}
\end{lemma}

\begin{proof}
Our proof of \eqref{eq:VarLowerBound} will depend on bounding
the second conditional moment $\E[d_{n,i}^2|\F_{i-1}]$ from below, and, since $X_{n,i}$
is $\F_{i-1}$-measurable, we see that when we condition on $\F_{i-1}$ the random variable
$d_{n,i}$ is a function of just $Y_i$ and $D_i$. If we call this function $f_{n,i}$, then we can also write
$$
d_{n,i}=f_{n,i}(Y_i,D_i) \quad \text{given } \F_{i-1}.
$$
We now consider a pair $(Y_i', D_i')$ such that $(Y_i', D_i')$
is independent of $(Y_i, D_i)$ and has the same distribution as $(Y_i, D_i)$.
The random variables $f_{n,i}(Y_i,D_i)$ and $f_{n,i}(Y_i',D_i')$ are therefore independent.  Moreover, conditional on the $\sigma$-field
$\F_{i-1}$ they have the same distribution as $d_{n,i}$ so if we set
$\Delta_i=f_{n,i}(Y_i,D_i) -f_{n,i}(Y_i',D_i')$ then by a familiar calculation we have
\begin{equation}\label{eq:MDS-squared-representation}
\E[d_{n,i}^2\,|\, \F_{i-1} ]
= \frac{1}{2} \E[ \Delta_i^2 \,|\, \F_{i-1} ].
\end{equation}

Using the definition \eqref{eq:MartDiffs} and the inventory recursion \eqref{eq:Xi-inventory} can now write
\begin{equation}\label{eq:Def-fn}
f_{n,i}(Y_i,D_i)=c(\gamma_{n,i}(X_{n,i})-X_{n,i})+g_{n,i}(Y_i,D_i)+h_{n,i}(Y_i,D_i)
\end{equation}
where
$$
g_{n,i}(Y_i,D_i)=Y_iL(\gamma_{n,i}(X_{n,i})-D_i) + (1-Y_i)L(X_{n,i}-D_i)
$$
and
$$
h_{n,i}(Y_i,D_i)=v_{n-i}(\gamma_{n,i}(X_{n,i})-D_i)-v_{n-i+1}(X_{n,i}).
$$

To work toward a simple lower bound for the conditional expectation \eqref{eq:MDS-squared-representation},
we now consider the event
$$
G_i=\big\{\, \omega: \,
D_{i}(\omega)\in [s_\infty, J], \,
D'_{i}(\omega) \in  [s_\infty, J], \,\, \text{and} \,\,
Y_i(\omega) = Y'_i(\omega)\, \big\},
$$
where $J$ is the supremum of the support of the demand density $\psi$ and $s_\infty$ is the top stocking level in \eqref{eq:base-stock-level-monotone}.

Since $X_{n,i} \leq s_\infty$ and $\gamma_{n,i}(X_{n,i}) \leq s_\infty$, on $G_i$ we have relations
\begin{equation}\label{eq:BigDBounds}
X_{n,i+1}=\gamma_{n,i}(X_{n,i})-D_i \leq 0 \quad \text{and} \quad X'_{n,i+1}=\gamma_{n,i}(X_{n,i})-D_i'\leq 0.
\end{equation}
Now, if we use $\1(A)$ to denote the indicator function of an event $A$, then
for all $i \in [1:n-n_0)$  we have from Lemma \ref{lm:Technical} that
\begin{equation}\label{eq:h-difference}
\{ h_{n,i}(Y_i,D_i)-h_{n,i}(Y_i',D_i') \} \1(G_i) =(D_i-D_i')\{ c+ c_p(1-q) \}\1(G_i).
\end{equation}

On $G_i$ we also have $Y_i=Y_i'$ so by \eqref{eq:BigDBounds}
and by the definition \eqref{eq:Def-L} of the carrying-cost function $L$, we have
\begin{align*}\label{eq:ExpDif}
\{ g_{n,i}(Y_i,D_i) & -g_{n,i}(Y_i',D_i')\}  \1(G_i)  \notag\\
& = Y_i\big\{\,  L\big(\gamma_{n,i}(X_{n,i}) -D_i\big)-L\big(\gamma_{n,i}(X_{n,i}) -D_i'\big)\,\big\} \1(G_i)\\
& \quad \quad+ (1-Y_i)\big\{\, L\big(X_{n,i} -D_i)-L(X_{n,i} -D_i'\big)\, \big\}\1(G_i)\notag\\
& = Y_i\{ -c_p\big(\gamma_{n,i}(X_{n,i}) -D_i)\big)+c_p\big(\gamma_{n,i}(X_{n,i}) -D_i'\big)\}\1(G_i)\notag\\
& \quad \quad +(1-Y_i)\{ -c_p\big(X_{n,i} -D_i\big)+c_p\big(X_{n,i} -D_i'\big)\}\1(G_i)\notag\\
& = c_p(D_i-D_i')\1(G_i).\notag
\end{align*}

From the last relation, the identity \eqref{eq:h-difference},
and the definition \eqref{eq:Def-fn} of $f_{n,i}$, we have
for all $i \in [1:n-n_0)$ that
\begin{equation*}
\{f_{n,i}(Y_i,D_i)-f_{n,i}(Y_i',D_i')\}\1(G_i)=(D_i-D_i') \{c+(2-q) c_p\}\1(G_i).
\end{equation*}
We also have $\Delta_i=\1(G_i) \Delta_i+\1(G_i^c) \Delta_i$ so
if we set $\rho \equiv \{c+(2-q) c_p\}>0$ we have
$
\Delta_i^2 \geq \1(G_i) \rho^2 (D_i - D'_i)^2.
$
Thus, by \eqref{eq:MDS-squared-representation} we see that for  all $i \in [1:n-n_0)$  we have the bound
\begin{align}\label{eq:CondExpLower}
\E[d_{n,i}^2\,|\, \F_{i-1} ]
&\geq  \frac{1}{2} \rho^2 \, \E[(D_i'-D_i)^2\1(G_i) \, | \, \F_{i-1} ]\\
&=\frac{1}{2} \rho^2 \, \int_{s_\infty}^J\int_{s_\infty}^J (u'-u)^2 \psi(u')\psi(u) \,du \equiv \beta'>0,\notag
\end{align}
and $\beta'$ is strictly positive because the first inequality of \eqref{eq:ProbDbelowSinf} implies that
the density $\psi$ cannot vanish identically on $[s_\infty, J]$.

Finally, using \eqref{eq:CondExpLower} in the range where it applies, we see that when
we take total expectations and sum we obtain
$$
\sum_{i=1}^n \E[d_{n,i}^2]\geq \sum_{i=1}^{n-n_0-1} \E[d_{n,i}^2]\geq \beta'(n-n_0-1).
$$
When we take $\beta=\beta'/2$ and $N=2(n_0+1)$, the proof of the lemma is complete.
\end{proof}

From our decomposition \eqref{eq:Cn-decomposition}
we have $\C_n(\pi^*_n)=\C_n'(\pi^*_n)+\C_n''(\pi^*_n)$ where we also have
$\norm{\C_n''(\pi^*_n)}_\infty \leq B(n_0+1)$ for a fixed constant $B$.
Thus, by the Cauchy-Schwarz inequality we find
\begin{equation}\label{eq:CSandVar}
\Var[\C_n(\pi^*_n)] =\Var[\C_n'(\pi^*_n)] +O(\sqrt{\Var[\C_n'(\pi^*_n)]}),
\end{equation}
and from Lemma \ref{lm:VarianceLowerBound}
we get that $\Var[\C_n'(\pi^*_n)] \sim \Var[\C_n(\pi^*_n)]$ as $ n\rightarrow \infty$.
This completes the second --- and last --- step in the
proof of the CLT for $\C_n'(\pi^*_n)$, i.e. we have
$$
\frac{\C_n'(\pi^*_n) - \E[\C_n'(\pi^*_n)]}{\sqrt{\Var[\C_n'(\pi^*_n)]}} \Longrightarrow N(0,1) \quad \quad \text{as } n \rightarrow \infty.
$$
Finally, by the boundedness of $\C_n''(\pi^*_n)$, it is trivial that
$$
\frac{\C_n''(\pi^*_n) - \E[\C_n''(\pi^*_n)]}{\sqrt{\Var[\C_n'(\pi^*_n)]}} \Longrightarrow 0 \quad \quad \text{as } n \rightarrow \infty,
$$
so, by summing the last two relations and using \eqref{eq:CSandVar} a second time,
we complete the proof of Theorem \ref{thm:inventory-CLT} by applying Theorem \ref{th:CLT-nonhomogeneous-chain} (or, more precisely, we apply the
special case given by
Corollary \ref{cor:CLT}).

\section{Further Connections and Open Problems} \label{se:conclusions}

In models where the optimality of a policy $\pi^*_n$ is defined purely in terms of expected costs there are good reasons to be
skeptical about the true economic value that is realized by the policy. A central limit theorem for the realized
cost, like our Theorem \ref{thm:inventory-CLT} for the
\citeauthor{Bulinskaya:TPA1964} inventory model, is one credible step toward the validation of the
mean-optimizing technology. Nevertheless, there are other steps one could take.

In particular, a concentration inequality can be used to make a similar point.
As we noted earlier in Remark \ref{re:BoundedMDS}, the martingale differences in our representation \eqref{eq:MDS-rep-Cn}
have the uniform boundedness property \eqref{eq:BddMDS}.
Consequently, by Hoeffding's inequality one finds for all $\lambda\geq 0$ that
\begin{equation}\label{eq:Hoeffding}
\P\big[\, \big|\C_n(\pi^*_n)-\E[\C_n(\pi^*_n)]\big|\, \geq \lambda \big]
 \leq 2 \exp\big\{-\lambda^2/(2nB)\, \big\}
\end{equation}
where
$B= \sup \{ \norm{d_{n,i}}_\infty :  1 \leq  i  \leq n \text{ and } 1 \leq n < \infty \}$.
This bound gives one an independent
argument that $\C_n(\pi_n^*)$ is well represented by $\E[\C_n(\pi^*_n)]$.
Specifically, it also tells us that deviations of order greater than
$O\{\E[\C_n(\pi^*_n)]^{1/2}\}=O(n^{1/2})$ have  small, well-specified, probabilities.

The proofs of both Theorem \ref{thm:inventory-CLT} and the concentration inequality \eqref{eq:Hoeffding}
make essential use of the assumption that the demand density $\psi$ has bounded support.
From a purely practical point of view, there is no problem with this assumption. Real-world demands are
bounded.
Nevertheless, in mathematical inventory theory one often allows for demand distributions that have unbounded support, so it is natural to ask if one has analogs of these results for such distributions.

The most direct approach to such an extension of Theorem \ref{thm:inventory-CLT} would be through the corresponding extension of the underlying CLT from
\citet{ArlSte:MOR2016} that we have used here. Such an extension
is credible, but it is challenging. Alternatively, one might try to extend Theorem \ref{thm:inventory-CLT} by
arguing from first principles, and here there is some encouragement from the unique
features of the random variable $\C_n(\pi_n^*)$. In particular, one might hope to exploit
the special structure of the martingale differences given by Proposition \ref{pr:OptMart},
but it is not easy to make more concrete suggestions.

On the other hand, since there are many flexible variations of Hoeffding's inequality,
there are sunnier prospects for concentration inequalities for $\C_n(\pi_n^*)$
when $\phi$ has unbounded support.
To be sure, such extensions are likely to lose the simplicity of \eqref{eq:Hoeffding},
but they could be useful in concrete contexts.

Finally, there is a third approach to the economic viability of
mean-optimal policies that can be built on the theory of stochastic orders.
In the past, the theory of stochastic orders has been used to explain how
properties of the transition kernels impact the the associated value functions
\citep[see, e.g.][]{Mueller:MOOR1997,SmithMcCardle:OR2002}, but here we have
something different in mind.

To be specific, we first
recall that a random variable $X$ is smaller than a random variable $Y$ in the (usual) stochastic
order if $\P(X>t) \leq \P(Y>t)$ for all $t \in \R$, and, in this case, we write $X \leq_{\rm st} Y$.
Now, if we consider the policy $\pi_n^*$ that is mean-optimal for the \citeauthor{Bulinskaya:TPA1964} model (or other stochastic optimization model), then it would be extraordinarily pleasing
if one could prove that
\begin{equation}\label{eq:SO-Problem}
\C_n(\pi^*_n) \leq_{\rm st} \C_n(\pi) \quad \text{for all } \pi \in \S
\end{equation}
where $\S$ is a class of feasible policies.

One of the many consequences that
would follow from \eqref{eq:SO-Problem} is that for each monotone increasing $\phi$ one would have
$$
\E[\phi(\C_n(\pi^*_n))] \leq  \E[\phi(\C_n(\pi))] \quad \text{for all } \pi \in \S.
$$
Since any utility function is monotone increasing, this would tell us that the mean-optimal
policy $\pi^*_n$ has \emph{minimal expected utility cost} for any utility.
To be sure, even this would not free one from the Saint Petersburg Paradox ---
essentially nothing does. Nevertheless, a bound like \eqref{eq:SO-Problem} would provide another important argument in the economic case for $\pi^*_n$.

For the moment, there is not quite enough evidence for \eqref{eq:SO-Problem} to put it forth as a conjecture. Still,
this kind of question is understudied,
and it provides a natural challenge. Moreover, there is considerable flexibility in how one might proceed.
For example, to take the investigation a step further,
one could consider any of the univariate orders in \citet{Shaked-Shanthikumar2007}
and ask again if one has the analog of \eqref{eq:SO-Problem} --- either for the \citeauthor{Bulinskaya:TPA1964} model, or for other
models that may be more tractable.

\enlargethispage{1cm}

\section*{Acknowledgement}

This material is based upon work supported by the National Science Foundation
under CAREER Award No. 1553274.

\bibliography{biblio-inventory}
\bibliographystyle{agsm}

\end{document}